\documentclass[11pt]{article}

\usepackage{amsmath,amssymb,amsfonts,amsthm}

\usepackage{caption}
\usepackage{subcaption}

\usepackage{tikz}
\usepackage{graphicx}
\usepackage{placeins}
\usetikzlibrary{automata}
\usetikzlibrary{arrows}
\usetikzlibrary{positioning,calc}
\usetikzlibrary{graphs}
\usetikzlibrary{graphs.standard}
\usetikzlibrary{arrows,decorations.markings}
\usepackage{tkz-graph}
\usetikzlibrary{chains,fit,shapes}
\usetikzlibrary{calc}
\tikzset{every loop/.style={min distance=10mm,looseness=10}}
\tikzset{every state/.style={minimum size=2mm}}

\newtheorem{theorem}{Theorem}

\newtheorem{lemma}[theorem]{Lemma}

\newtheorem{proposition}[theorem]{Proposition}

\newtheorem{Fact}[theorem]{Observation}

\title{On the representation number of chessboard graphs}

\author{Sergey Kitaev\footnote{Department of Mathematics and Statistics, University of Strathclyde, 26 Richmond Street, Glasgow G1, 1XH, United Kingdom. 
{\bf Email:} sergey.kitaev@strath.ac.uk.}\ \ and Artem Pyatkin\footnote{Sobolev Institute of Mathematics, Koptyug ave, 4, Novosibirsk, 630090, Russia\  {\bf Email:} artem@math.nsc.ru.}}

\begin{document}
	\maketitle	
	
\begin{abstract}
The representation number of a graph is the smallest integer $k$ such that the graph can be represented by a word in which each vertex appears exactly $k$ times, and two distinct vertices $x$ and $y$ alternate in the word if and only if they are adjacent in the graph.

We extend known results on the representation number for various graph classes to chessboard graphs—namely, king, queen, rook, bishop, and knight graphs. We provide a complete classification for queen graphs and partial classifications or observations for the other classes. As a consequence of our study, we obtain a characterization of all chessboard graphs that are circle graphs. Our work also leads to several interesting open problems.\\

\end{abstract}	

\section{Introduction}

Graphs defined by the legal moves of chess pieces on a square or rectangular board — such as those arising from the movements of kings, queens, rooks, bishops, or knights — have been widely studied in discrete mathematics and theoretical computer science. These graphs are commonly referred to in the literature as \emph{chess graphs}~\cite{Cibu2024} or \emph{chessboard graphs}~\cite{SowndaryaNaidu2020}, especially in the context of classic problems involving independence, domination, coloring, or connectivity on rectangular grids. Another common phrasing is \emph{graphs induced by chess piece movements}~\cite{Haynes1998}, which emphasizes the rule-based construction of edges between board positions based on each piece's allowed moves. 

\emph{Word-representable graphs}~\cite{KL15,KP18} generalize various graph families such as 3-colourable graphs, comparability graphs, and circle graphs. The roots of word-representable graphs lie in~\cite{KitSei}, where similar graphs not only allowed the determination of the \emph{free spectrum} of the celebrated \emph{Perkins semigroup} (which has played a central role in semigroup theory since 1960, particularly as a source of examples and counterexamples), but also enabled the solution of the \emph{word problem} for this semigroup. The first systematic study of word-representable graphs appears in~\cite{KitPya08}, and in the last decade, the topic has attracted much attention in the literature. Notably, these graphs were studied not only from the graph-theoretic perspective but also from the word-theoretic perspective, where classes of words were considered and the graphs represented by them were investigated \cite{Fleisch}. The \emph{representation number} of a graph is the minimum number of copies of each letter required to represent the graph~\cite{KL15,KP18} (see below for formal definitions).

In this paper, we extend recent studies on the representation number of grid graphs and cylindric grid graphs~\cite{AKP2025} to chessboard graphs --- specifically, to king, queen, rook, bishop, and knight graphs. We provide a complete classification for queen graphs and partial classifications for the remaining classes, noting that bishop graphs are essentially equivalent to rook graphs in terms of word-representability. Our results also yield a complete characterization of chessboard graphs that are circle graphs. Finally, this work gives rise to several interesting open problems.

\subsection{Word-representable graphs and relevant results}

A graph $G=(V,E)$ is {\em word-representable} if there exists a word $w$ over the alphabet $V$ such that two distinct letters $x$ and $y$ alternate in $w$ if and only if $(x,y)\in E$. It is known \cite{KL15,KP18} that any word-representable graph $G$  is {\em $k$-word-representable} for some $k$; that is, there exists  a word $w$ representing $G$ in which each letter appears exactly $k$ times. The minimum such $k$ is called the {\em representation number} of $G$, and it is denoted by $\mathcal{R}(G)$. Clearly,  $\mathcal{R}(G)=1$ if and only if $G$ is a complete  graph 
$K_n$ for some $n$.
 For a non-word-representable graph $G$, by definition, $\mathcal{R}(G)=\infty$.

The representation number of graphs has been studied in the literature for various classes of graphs --- for example,  {\em crown graphs} \cite{GKP18,HHMO24}, the {\em $k$-dimensional cube} \cite{Bro18,BroZan19,HHMO24}, and {\em grid graphs} and {\em cylindric grid graphs} \cite{AKP2025}. An $m\times n$ {\em grid graph}  $\mbox{Gr}_{m,n}$ has vertex set $V=\{x_{ij}\ |\ i=1,\ldots,m, j=1,\ldots,n\}$ and edge set 
$E=\{(x_{ij},x_{i,j+1})\ |\  i=1,\ldots,m, j=1,\ldots,n-1\} \cup \{(x_{ij},x_{i+1,j})\ |\  i=1,\ldots,m-1, j=1,\ldots,n\}$. A {\em cylindric grid graph} $\mbox{CGr}_{m,n}$ is obtained from $\mbox{Gr}_{m,n}$ by adding edges $(x_{i,1},x_{i,n})$ for all $i=1,\ldots,m$, where we assume that $n\geq 3$. A {\em toroidal grid graph} $\mbox{TGr}_{m,n}$ is obtained from the cylindric grid graph $\mbox{CGr}_{m,n}$ by adding edges $(x_{1,j},x_{m,j})$ for all $j=1,\ldots,n$, where we assume $m,n\geq 3$. In other words, $\mbox{Gr}_{m,n} = P_m \Box P_n$, $\mbox{CGr}_{m,n} = P_m \Box C_n$, and $\mbox{TGr}_{m,n} = C_m \Box C_n$, where $G \Box H$ denotes the Cartesian product of graphs $G$ and $H$, and $P_n$ (resp., $C_n$) is the path (resp., cycle) on $n$ vertices. Throughout the paper, w.l.o.g., we assume $m \le n$ because of the symmetry of the chessboard graphs.

It is known~\cite{KL15,KP18} that the class of \emph{circle graphs}, excluding complete graphs, coincides precisely with the class of graphs having representation number~2. In particular, $\mathcal{R}(P_n) = \mathcal{R}(\mbox{Gr}_{1,n}) = 2$ and $\mathcal{R}(C_n) = \mathcal{R}(\mbox{CGr}_{1,n}) = 2$ for $n \geq 3$. 

Of importance to our work are studies on graphs with representation number~3, as conducted in~\cite{AKP2025,Kit13,KitPya08} (see also Section~5.2 in~\cite{KL15} for a summary). In particular, the \emph{$n$-prism graph} $\mbox{Pr}_n=C_n\mathbin{\Box}P_2$ has vertex set
$\{u_1,\ldots,u_n,v_1,\ldots,v_n\}$ and edge set
\[
\{u_iu_{i+1},v_iv_{i+1},u_iv_i:1\leq i\leq n\},
\]
where subscripts are taken modulo~$n$. Thus it consists of two cycles $C_n$ joined by a matching between corresponding vertices. Every prism graph has representation number~3.

Moreover, it is known that the representation number of a ladder graph is 2~\cite{Kit13}; thus, $\mathcal{R}(\mbox{Gr}_{2,n}) = 2$ (a ladder graph) for $n \geq 2$, and $\mathcal{R}(\mbox{Gr}_{2,1}) = \mathcal{R}(\mbox{Gr}_{1,1}) = 1$ (a complete graph). For $m,n \geq 3$, we have $\mathcal{R}(\mbox{Gr}_{m,n}) = 3$ and $\mathcal{R}(\mbox{CGr}_{m,n}) = 3$ \cite{AKP2025}.

All of the chessboard graphs considered in this paper and defined below contain grid graphs as (non-induced) subgraphs, with the exception of knight graphs. Thus, the research presented here can be seen as an extension of the results in~\cite{AKP2025,Kit13} to interesting classes of graphs that have been studied in the literature from various perspectives.

\subsection{King  graphs} A king  graph is a graph that represents all legal moves of the king chess piece on a chessboard, where each vertex corresponds to a square on the board and each edge corresponds to a legal move~\cite{Ma,VWSFN}. More specifically, the {\em king  graph} $\mathrm{King}_{m,n}$ has the vertex set $\{ x_{i,j} \mid i=1,\ldots,m,\ j=1,\ldots,n \}$ and an edge between distinct vertices $x_{a,b} \ne x_{c,d}$ if and only if $|a - c| \le 1$ and $|b - d| \le 1$.

For a fixed $j$, we call the subgraph of $\mathrm{King}_{m,n}$ induced by the vertices $x_{i,j}$, $i=1,\ldots,m$, the {\em $j^{\rm th}$ layer} of $\mathrm{King}_{m,n}$. Each layer is the path $P_m$. The graph $\mbox{King}_{3,5}$ is shown in Figure~\ref{3x5-torus-grid-graph} on the left.

The enumeration of the different placements of the maximum number of non-attacking kings on a chessboard in~\cite{Wilf}, that is, placements in which no two squares containing kings share a point, is directly related to counting maximal independent sets in king graphs. Our results on king  graphs, proved in Section~\ref{king-sec}, are summarized in Table~\ref{kings-results}.%, where we assume, w.l.o.g., that $n\geq m$  (since $\mathcal{R}(\mbox{King}_{m,n})=\mathcal{R}(\mbox{King}_{n,m}))$.

\begin{table}[t]
  \centering
  \begin{subtable}[t]{0.48\textwidth}
    \centering
    \begin{tabular}{|c||c|c|c|c|}
      \hline
      $m \backslash n$ & 1 & 2 & 3 & $\geq 4$ \\
      \hline\hline
	      1 &   1 &  1 &  2 &  2 \\
      \hline
	      2 &  1 &  1 & 2  &  2 \\
      \hline
      3 & 2  &  2 &  3 &  3 \\
      \hline
      4 & 2  & 2  &  3  &  3 \\
      \hline
    \end{tabular}
    \caption{The number $\mathcal{R}(\mathrm{King}_{m,n})$, which is finite for all $m, n \geq 1$.}
    \label{kings-results}
  \end{subtable}
  \hfill
  \begin{subtable}[t]{0.48\textwidth}
    \centering
    \begin{tabular}{|c||c|c|c|c|c|}
      \hline
      $m \backslash n$ & 1 & 2 & 3 & 4 & $\geq 5$ \\
      \hline\hline
      1 &  1  & 1  & 1  & 1 & 1 \\
      \hline
      2 & 1  &  1  &   2 & 2  & 3 \\
      \hline
      3 & 1  &  2 &  $\infty$ &   $\infty$ &   $\infty$ \\
      \hline
      4 & 1  &  2 &  $\infty$ &   $\infty$ &   $\infty$ \\
      \hline
      $\geq 5$ & 1  &  3 &  $\infty$ &   $\infty$ &   $\infty$ \\
      \hline
    \end{tabular}
    \caption{The number $\mathcal{R}(\mbox{Queen}_{m,n})$.}
    \label{queen-results}
  \end{subtable}
  \caption{Representation numbers of king  and queen graphs.}
  \label{fig:king-queen-tables}
\end{table}

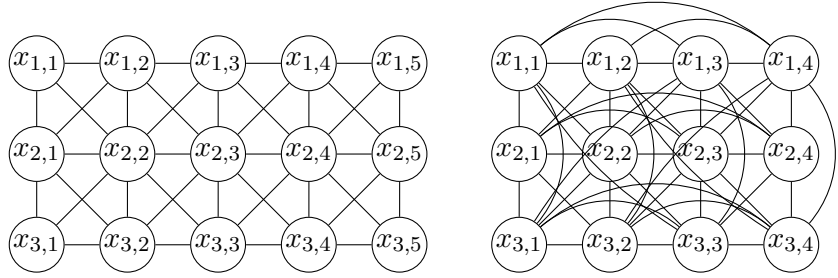
\begin{figure}
\begin{center}
\begin{tabular}{ccc}
\begin{tikzpicture}[scale=1.2, every node/.style={circle, draw, inner sep=0.5pt}]
  % Define number of rows and columns
  \def\rows{3}
  \def\cols{5}

  % Draw vertices
  \foreach \i in {1,2,3} {
    \foreach \j in {1,2,3,4,5} {
      \node (x\i\j) at (\j,-\i) {$x_{\i,\j}$};
    }
  }

  % Draw horizontal edges
  \foreach \i in {1,2,3} {
    \foreach \j in {1,2,3,4} {
      \draw (x\i\j) -- (x\i\the\numexpr\j+1\relax);
    }
  }

  % Draw vertical edges
  \foreach \i in {1,2} {
    \foreach \j in {1,2,3,4,5} {
      \draw (x\i\j) -- (x\the\numexpr\i+1\relax\j);
    }
  }
  
\draw (x11) to (x22);  \draw (x12) to (x23);  \draw (x13) to (x24);  \draw (x14) to (x25);  
\draw (x21) to (x32);  \draw (x22) to (x33);  \draw (x23) to (x34);  \draw (x24) to (x35);   
\draw (x21) to (x12);  \draw (x22) to (x13);  \draw (x23) to (x14);  \draw (x24) to (x15);  
\draw (x31) to (x22);  \draw (x32) to (x23);  \draw (x33) to (x24);  \draw (x34) to (x25);

\end{tikzpicture}

& & 

\begin{tikzpicture}[scale=1.2, every node/.style={circle, draw, inner sep=0.5pt}]
  % Define number of rows and columns (manually adjust since column 5 is removed)
  \def\rows{3}
  \def\cols{4}

  % Draw vertices (excluding j = 5)
  \foreach \i in {1,2,3} {
    \foreach \j in {1,2,3,4} {
      \node (x\i\j) at (\j,-\i) {$x_{\i,\j}$};
    }
  }

  % Draw horizontal edges
  \foreach \i in {1,2,3} {
    \foreach \j in {1,2,3} {
      \draw (x\i\j) -- (x\i\the\numexpr\j+1\relax);
    }
  }

  % Draw vertical edges
  \foreach \i in {1,2} {
    \foreach \j in {1,2,3,4} {
      \draw (x\i\j) -- (x\the\numexpr\i+1\relax\j);
    }
  }

  % Draw diagonal edges (excluding any involving column 5)
  \draw (x11) -- (x22);  \draw (x12) -- (x23);  \draw (x13) -- (x24);
  \draw (x21) -- (x32);  \draw (x22) -- (x33);  \draw (x23) -- (x34);
  \draw (x21) -- (x12);  \draw (x22) -- (x13);  \draw (x23) -- (x14);
  \draw (x31) -- (x22);  \draw (x32) -- (x23);  \draw (x33) -- (x24);

  % The four length-two diagonal moves.
  \draw[bend right=12] (x11) to (x33);
  \draw[bend right=12] (x12) to (x34);
  \draw[bend left=12] (x31) to (x13);
  \draw[bend left=12] (x32) to (x14);

  % Curved non-oriented arcs within each row (excluding adjacent vertices)
  \foreach \r in {1,2,3} {
    \foreach \i in {1,2} {
      \foreach \j in {\the\numexpr\i+2,\the\numexpr\i+3} {
        \ifnum\j<5
          \draw[bend left=40] (x\r\i) to (x\r\j);
        \fi
      }
    }
  }

\foreach \j in {1,2,3,4} {
  \draw[bend left=40] (x1\j) to (x3\j);
}

\end{tikzpicture}
\end{tabular}

\caption{The king  graph $\mbox{King}_{3,5}$ and the queen graph $\mbox{Queen}_{3,4}$.}\label{3x5-torus-grid-graph}
\end{center}
\end{figure}

\subsection{Queen graphs} A {\em queen graph} is an undirected graph that represents all legal moves of the queen chess piece on a chessboard~\cite{HHHH}. In the graph, each vertex corresponds to a square on the chessboard, and an edge connects two vertices if a queen can move between the corresponding squares; that is, if the squares lie in the same row, column, or diagonal. If the chessboard has dimensions $m \times n$, the resulting graph is % called the $m \times n$ queen graph, 
denoted $\mbox{Queen}_{m,n}$. The graph $\mbox{Queen}_{3,4}$ is shown in Figure~\ref{3x5-torus-grid-graph} on the right.  Our results on queen graphs, proved in Section~\ref{queen-sec}, are summarized in Table~\ref{queen-results}.%, where we assume, w.l.o.g., that $n\geq m$  (since $\mathcal{R}(\mbox{Queen}_{m,n})=\mathcal{R}(\mbox{Queen}_{n,m}))$.

\subsection{Rook  graphs} A \emph{rook  graph}, denoted $\mathrm{Rook}_{m,n}$, is an undirected graph representing all legal moves of the rook chess piece on an $m \times n$ chessboard~\cite{BLL}. Each vertex corresponds to a square of the chessboard, and there is an edge between any two vertices that lie in the same row or the same column.
%---exactly the squares between which a rook can move. 
Rook  graphs are significant in graph theory due to their alternative constructions: they are Cartesian products of two complete graphs (and hence word-representable~\cite{KL15}), and they are also line graphs of complete bipartite graphs. The square rook  graphs form the class of \emph{two-dimensional Hamming graphs}. The graph $\mathrm{Rook}_{4,4}$ is shown in Figure~\ref{4x4-rook-graph} on the left.

Note that the toroidal graph $\mathrm{TGr}_{m,n}$ is a (non-induced if $\max\{m,n\} \geq 4$) subgraph of $\mathrm{Rook}_{m,n}$, while $\mathrm{TGr}_{3,3} = \mathrm{Rook}_{3,3}$. Our results on rook  graphs, proved in Section~\ref{rook-sec}, are summarized in Table~\ref{rook-results}.

\begin{table}[t]
  \centering
  \begin{subtable}[t]{0.47\textwidth}
    \centering
    \begin{tabular}{|c||c|c|c|c|}
      \hline
      $m \backslash n$ & 1 & 2 & 3 & $\geq 4$ \\
      \hline\hline
      1 &   1 &  1 &  1 &  1 \\
      \hline
      2 &  1 &  2 &  3  &  3 \\
      \hline
	      3 &  1 &  3 &  3 &  $\mathrm{open}^{*}$ \\
      \hline
    \end{tabular}
	    \caption{The number $\mathcal{R}(\mbox{Rook}_{m,n})$, which is finite for all $m, n \geq 1$. The asterisk marks a case in which computer search suggests a lower bound of~4, but no proof is known.}
    \label{rook-results}
  \end{subtable}
  \hfill
  \begin{subtable}[t]{0.47\textwidth}
    \centering
    \begin{tabular}{|c||c|c|c|c|}
      \hline
      $m \backslash n$ & 1 & 2 & 3 & $\geq 4$ \\
      \hline\hline
	      1 &   1 &  2 &  2 &  2 \\
      \hline
      2 &  2 &  2 &  2  &  2 \\
      \hline
      3 &  2 &  2 &  2 &  2 \\
      \hline
      4 &  2 &  2 &  2 &  $\geq 3$ \\
      \hline
    \end{tabular}
    \caption{The number $\mathcal{R}(\mbox{Bishop}_{m,n})$, which is finite for all $m, n \geq 1$.}
    \label{bishop-results}
  \end{subtable}
  \caption{Representation numbers of rook  and bishop  graphs.}
  \label{fig:rook-bishop-tables}
\end{table}

\subsection{Bishop  graphs} A {\em bishop  graph}, denoted $\mbox{Bishop}_{m,n}$, is an undirected graph that encodes all legal moves of the bishop chess piece on an $m \times n$ chessboard \cite{Bur2016,SowndaryaNaidu2020}. Each vertex corresponds to a square of the chessboard, and two vertices are adjacent if  
%the bishop can move between the corresponding squares—that is, if 
they lie on the same diagonal. Unlike other pieces, bishop graphs are disconnected because bishops are restricted to one color class of the board, and each connected component of $\mbox{Bishop}_{m,n}$ corresponds to one color class. An example of the graph $\mbox{Bishop}_{4,4}$ is shown in Figure~\ref{4x4-rook-graph} on the right. Our results on bishop  graphs, proved in Section~\ref{bishop-sec}, are summarized in Table~\ref{bishop-results}.%, where we assume, w.l.o.g., that $n\geq m$  (since $\mathcal{R}(\mbox{Bishop}_{m,n})=\mathcal{R}(\mbox{Bishop}_{n,m}))$.

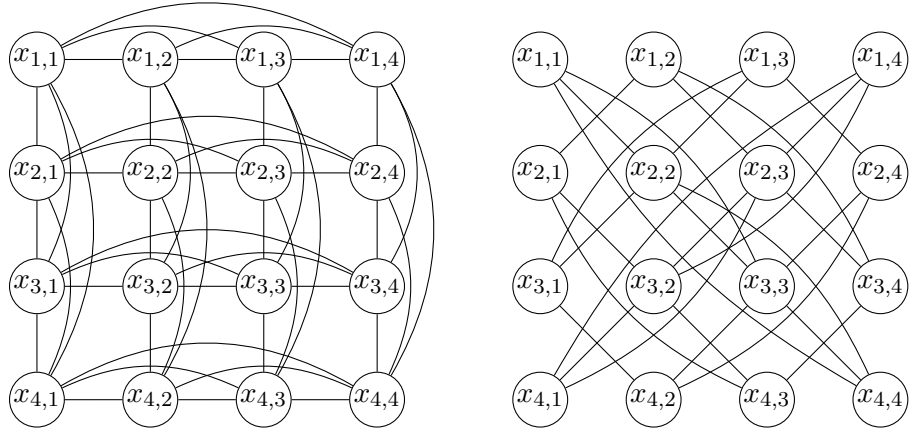
\begin{figure}
\begin{center}

\begin{tabular}{ccc}

\begin{tikzpicture}[scale=1.5, every node/.style={circle, draw, inner sep=0.5pt}]
  % Define number of rows and columns
  \def\rows{3}
  \def\cols{5}

  % Draw vertices
  \foreach \i in {1,2,3,4} {
    \foreach \j in {1,2,3,4} {
      \node (x\i\j) at (\j,-\i) {$x_{\i,\j}$};
    }
  }

  % Draw horizontal edges
  \foreach \i in {1,2,3,4} {
    \foreach \j in {1,2,3} {
      \draw (x\i\j) -- (x\i\the\numexpr\j+1\relax);
    }
  }

  % Draw vertical edges
  \foreach \i in {1,2,3} {
    \foreach \j in {1,2,3,4} {
      \draw (x\i\j) -- (x\the\numexpr\i+1\relax\j);
    }
  }
  
  % Draw curved edges: first \UTF{2194} last vertex in each row
\draw[bend left=25] (x11) to (x13); \draw[bend left=30] (x11) to (x14); \draw[bend left=25] (x12) to (x14);% top row
\draw[bend left=25] (x21) to (x23); \draw[bend left=30] (x21) to (x24); \draw[bend left=25] (x22) to (x24);% second row
\draw[bend left=25] (x31) to (x33); \draw[bend left=30] (x31) to (x34); \draw[bend left=25] (x32) to (x34);% third row
\draw[bend left=25] (x41) to (x43); \draw[bend left=30] (x41) to (x44); \draw[bend left=25] (x42) to (x44);% fourth row

  % Draw curved edges: top \UTF{2194} bottom vertex in each column
\draw[bend left=25] (x11) to (x31); \draw[bend left=25] (x21) to (x41); \draw[bend left=30] (x11) to (x41);% first column
\draw[bend left=30] (x12) to (x32); \draw[bend left=25] (x22) to (x42); \draw[bend left=30] (x12) to (x42);% second column
\draw[bend left=30] (x13) to (x33); \draw[bend left=25] (x23) to (x43); \draw[bend left=30] (x13) to (x43);% third column
\draw[bend left=30] (x14) to (x34); \draw[bend left=25] (x24) to (x44); \draw[bend left=30] (x14) to (x44);% fourth column
 
\end{tikzpicture}

&&

\begin{tikzpicture}[scale=1.5, every node/.style={circle, draw, inner sep=0.5pt}]
  % Draw vertices
  \foreach \i in {1,2,3,4} {
    \foreach \j in {1,2,3,4} {
      \node (x\i\j) at (\j,-\i) {$x_{\i,\j}$};
    }
  }

  % Draw bishop diagonal edges: top-left to bottom-right (i - j = const)
  \foreach \i in {1,2,3} {
    \foreach \j in {1,2,3} {
      \pgfmathtruncatemacro{\k}{\i+1}
      \pgfmathtruncatemacro{\l}{\j+1}
      \ifnum\k<5
        \ifnum\l<5
          \draw (x\i\j) -- (x\k\l);
        \fi
      \fi
    }
  }

  \foreach \i in {2,3,4} {
    \foreach \j in {1,2,3} {
      \pgfmathtruncatemacro{\k}{\i-1}
      \pgfmathtruncatemacro{\l}{\j+1}
      \ifnum\k>0
        \ifnum\l<5
          \draw (x\i\j) -- (x\k\l);
        \fi
      \fi
    }
  }

  % Curved bishop edges for longer diagonals (i-j and i+j = const)
  \draw[bend left=20] (x11) to (x33);
  \draw[bend left=20] (x12) to (x34);
  \draw[bend left=-20] (x21) to (x43);
  \draw[bend left=-20] (x13) to (x31);
  \draw[bend left=20] (x24) to (x42);
  \draw[bend left=20] (x14) to (x32);
  \draw[bend left=20] (x22) to (x44);
  \draw[bend left=-20] (x41) to (x23);
\draw[bend left=-15] (x11) to (x44);
\draw[bend left=-15] (x14) to (x41);

\end{tikzpicture}

\end{tabular}

\caption{The rook  graph $\mbox{Rook}_{4,4}$ and the bishop  graph $\mbox{Bishop}_{4,4}$.}\label{4x4-rook-graph}
\end{center}
\end{figure}

\subsection{Knight  graphs}

A {\em knight  graph}, denoted $\mbox{Knight}_{m,n}$, is an undirected graph that represents all legal moves of the knight chess piece on an $m \times n$ chessboard \cite{Watkins1996}. Each vertex corresponds to a square on the board, and two vertices are adjacent if %a knight can move between the corresponding squares in a single legal move — that is, in an 
the corresponding squares form an L-shaped pattern: two steps in one direction (horizontal or vertical) followed by one step orthogonal to that direction. 
Note that knight graphs are always bipartite and they are disconnected for $m\le 2$ and $m=n=3$. The graph $\mbox{Knight}_{4,4}$ is shown in Figure~\ref{4x4-knight-graph}. Our results on knight graphs, proved in Section~\ref{knight-sec}, are summarized in Table~\ref{knight-results}.%, where we assume, w.l.o.g., that $n\geq m$  (since $\mathcal{R}(\mbox{Knight}_{m,n})=\mathcal{R}(\mbox{Knight}_{n,m}))$.

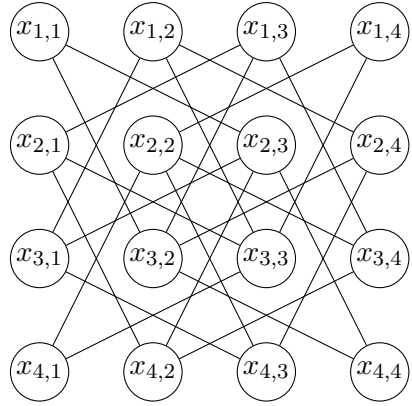
\begin{figure}[ht]
\begin{center}
\begin{tikzpicture}[scale=1.5, every node/.style={circle, draw, inner sep=1pt}]
  % Create 4x4 grid of vertices
  \foreach \i in {1,2,3,4} {
    \foreach \j in {1,2,3,4} {
      \node (x\i\j) at (\j,-\i) {$x_{\i,\j}$};
    }
  }

  % Add knight move edges manually using relative offsets
  \foreach \i in {1,2,3,4} {
    \foreach \j in {1,2,3,4} {
      \foreach \di/\dj in {2/1, 2/-1, 1/2, 1/-2} {
        \pgfmathtruncatemacro{\ni}{\i+\di}
        \pgfmathtruncatemacro{\nj}{\j+\dj}
        \ifnum\ni>0 \ifnum\ni<5 \ifnum\nj>0 \ifnum\nj<5
          \draw (x\i\j) -- (x\ni\nj);
        \fi\fi\fi\fi
      }
    }
  }
\end{tikzpicture}
\caption{The knight graph $\mathrm{Knight}_{4,4}$.}\label{4x4-knight-graph}
\end{center}
\end{figure}

\begin{figure}[ht]
\begin{center}
\begin{tikzpicture}[scale=1.05, every node/.style={circle, draw, inner sep=1.2pt}]
  \node (x11) at (-3,1) {$x_{1,1}$};
  \node (x12) at (-1,1.6) {$x_{1,2}$};
  \node (x13) at (1,1) {$x_{1,3}$};
  \node (x23) at (3,0) {$x_{2,3}$};
  \node (x33) at (1,-1) {$x_{3,3}$};
  \node (x32) at (-1,-1.6) {$x_{3,2}$};
  \node (x31) at (-3,-1) {$x_{3,1}$};
  \node (x21) at (-4,0) {$x_{2,1}$};

  % Consecutive vertices on the boundary.
  \draw (x11)--(x12)--(x13)--(x23)--(x33)--(x32)--(x31)--(x21)--(x11);
  % Remaining row and column edges, routed outside the centre where possible.
  \draw[bend left=0] (x11) to (x13);
  \draw[bend right=0] (x31) to (x33);
  \draw[bend right=18] (x11) to (x31);
  \draw[bend left=18] (x13) to (x33);
  \draw (x21)--(x23);
  \draw[bend left=22] (x12) to (x32);
  % Diagonal edges.
  \draw[bend left=0] (x11) to (x33);
  \draw[bend right=16] (x13) to (x31);
  \draw[bend right=-16] (x12) to (x21);
  \draw[bend left=-16] (x12) to (x23);
  \draw[bend right=16] (x32) to (x21);
  \draw[bend right=16] (x23) to (x32);
\end{tikzpicture}
\caption{The neighbourhood of $x_{2,2}$ in $\mathrm{Queen}_{3,3}$.}\label{queen-neighborhood}
\end{center}
\end{figure}
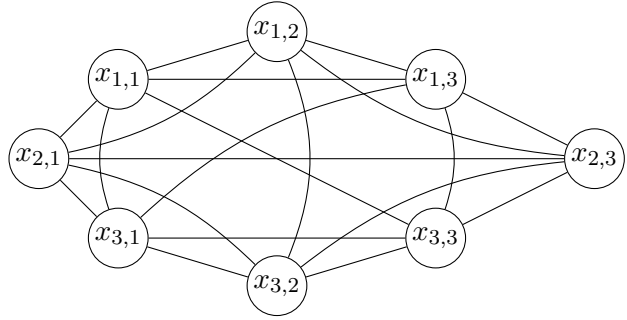

\begin{table}[t]
\begin{center}
\begin{tabular}{|c||c|c|c|c|c|}
\hline
$m \backslash n$ & 1 & 2 & 3 & 4 & $\geq 5$ \\
\hline
\hline
	1 &   1 &  2 &  2 &  2 & 2 \\
\hline
2 &  2 &  2 &  2  &  2  & 2 \\
\hline
3 &  2 &  2 &  2 &  2 & $\geq 3$\\
\hline
4 &  2 &  2 &  2 &  $\geq 3$ & $\geq 3$ \\
\hline
\end{tabular}
\end{center}
\caption{The number $\mathcal{R}(\mbox{Knight}_{m,n})$, which is finite for all $m, n \geq 1$.}\label{knight-results}
\end{table}

\section{Results and Proofs}

\subsection{King  graphs}\label{king-sec}

It follows from \cite{AKM} (see also \cite[Theorem~4.6.5]{KL15}) that $\mbox{King}_{m,n}$ is word-representable, and hence $\mathcal{R}(\mbox{King}_{m,n})<\infty$ for any $m,n\geq 1$. Clearly, $\mathcal{R}(\mbox{King}_{1,1})= \mathcal{R}(\mbox{King}_{1,2})=\mathcal{R}(\mbox{King}_{2,2})=1$ and $\mathcal{R}(\mbox{King}_{1,n})=\mathcal{R}(P_n)=2$ for $n\geq 3$.

\begin{theorem}\label{King-2n-thm} We have $\mathcal{R}(\mathrm{King}_{2,n})=2$ for $n\geq 3$. \end{theorem}

\begin{proof} Since $\mbox{King}_{2,n}$ is not complete, $\mathcal{R}(\mbox{King}_{2,n})>1$. Consider an arbitrary 2-representation $w$ of the path $P_n$ (see, e.g. \cite[Sections 3.1 and 3.2]{KL15}) to represent the top row in $\mbox{King}_{2,n}$. The representation of $\mbox{King}_{2,n}$ is then obtained from $w$ by replacing each letter $x_{1,i}$ by $x_{1,i}x_{2,i}$. Indeed, $(x_{1,i},x_{2,i})$ is an edge in $\mbox{King}_{2,n}$,
and a vertex $v$ is adjacent to $x_{2,i}$ if and only if $v$ is adjacent to $x_{1,i}$ for every $v\not\in \{x_{1,i},x_{2,i}\}$. \end{proof}

The proof of the following lemma is analogous to the proof of Lemma 3 in \cite{AKP2025}.

\begin{lemma}\label{lem-Gr33} We have $\mathcal{R}(\mathrm{King}_{3,3})\geq 3$. \end{lemma} 

\begin{proof}
Since $\mathrm{King}_{3,3}$ is not complete, $\mathcal{R}(\mathrm{King}_{3,3}) > 1$; therefore, it is sufficient to prove that $\mathcal{R}(\mathrm{King}_{3,3}) \neq 2$.

Consider the subset $N=\{x_{1,1}, x_{1,3}, x_{3,1}, x_{3,3}\}$ of pairwise non-adjacent vertices, each of which is adjacent to $x_{2,2}$. Suppose that $\mathrm{King}_{3,3}$ is 2-represented by a word $w$. After a cyclic shift, the restriction of $w$ to $N\cup\{x_{2,2}\}$ has the form
\[
x_{2,2}\,\pi\,x_{2,2}\,\pi^{\mathrm{rev}},
\]
where $\pi$ is a permutation of the four letters in $N$. Indeed, alternation with $x_{2,2}$ places one occurrence of each letter of $N$ between the two occurrences of $x_{2,2}$ and the other outside them. Moreover, since the letters of $N$ are pairwise non-adjacent, their relative order in the two blocks must be reversed. The eight symmetries of the square have three orbits on the $24$ possible orders $\pi$; choosing one representative from each orbit gives the following three possibilities:
\vspace{-2mm}
\begin{align*}
&\boxed{x_{2,2}}x_{1,1} x_{1,3} x_{3,1} x_{3,3}\boxed{x_{2,2}} x_{3,3} x_{3,1} x_{1,3} x_{1,1},\\
&\boxed{x_{2,2}}x_{1,1} x_{1,3} x_{3,3} x_{3,1}\boxed{x_{2,2}} x_{3,1} x_{3,3} x_{1,3} x_{1,1},\\
&\boxed{x_{2,2}}x_{1,1} x_{3,3} x_{1,3} x_{3,1}\boxed{x_{2,2}} x_{3,1} x_{1,3} x_{3,3} x_{1,1}.
\end{align*}

In the first case, one copy of $x_{2,3}$ must lie between the two occurrences of $x_{3,3}$; then the other copy cannot be between those of $x_{1,3}$. Hence, $x_{2,3}$ and $x_{3,1}$ must be adjacent in $\mathrm{King}_{3,3}$, which is a contradiction. In the remaining two cases, one copy of $x_{2,1}$ must lie between the occurrences of $x_{3,1}$, and the second one --- outside the two occurrences of $x_{1,1}$. Thus, $x_{2,1}$ and $x_{1,3}$ must be adjacent in $\mathrm{King}_{3,3}$, again a contradiction.

Therefore, $\mathcal{R}(\mathrm{King}_{3,3}) \geq 3$.
\end{proof}

If a letter $x$ occurs several times in a word $w$, denote by $x^{(j)}$ the $j^{\rm th}$ occurrence of $x$ in $w$. In what follows, a \emph{factor} of a word is a contiguous subword. We write \(x^{(i)} < y^{(j)}\) if the indicated occurrence of \(x\) appears to the left of the indicated occurrence of \(y\) in $w$. In this case, \([x^{(i)},y^{(j)}]_w\) denotes the factor of $w$ beginning with \(x^{(i)}\) and ending with \(y^{(j)}\), with both endpoints included. For a set $A$ of letters, $w|_A$ denotes the subword obtained from $w$ by deleting all letters not in $A$; the same notation is used for the restriction of a factor. Thus expressions such as $[x^{(i)},y^{(j)}]_w|_A$ are restrictions of factors, not set-theoretic intersections. In a uniform representing word, a letter $y$ alternates with $x$ precisely when exactly one occurrence of $y$ lies strictly between every two consecutive occurrences of $x$.
The following observation follows directly from the definitions.

\begin{Fact}~\label{evident}
If $(x,y)\in E$, then $x^{(i)}>y^{(i-1)}$ and $y^{(i)}>x^{(i-1)}$ for every $i\geq 2$ for which the indicated occurrences are defined.
\end{Fact}

Next, we prove that $\mathrm{King}_{3,n}$ and 
$\mathrm{King}_{4,n}$ are 3-representable. Since $\mathrm{King}_{3,n}$ is an induced subgraph of $\mathrm{King}_{4,n}$, it is sufficient to provide a 3-representation of
$\mathrm{King}_{4,n}$ only.  Let
\begin{align*}
&\texttt{Od} = x_1x_2x_1 x_3 x_2 x_1 x_4 x_3 x_2 x_4 x_3 x_4\\
&= x_1^{(1)} x_2^{(1)} x_1^{(2)} x_3^{(1)} x_2^{(2)} x_1^{(3)} x_4^{(1)} x_3^{(2)} x_2^{(3)} x_4^{(2)} x_3^{(3)} x_4^{(3)},
\end{align*}
and
\begin{align*}
&\texttt{Ev} = x_2 x_1 x_3 x_2 x_1 x_4 x_3 x_2 x_4 x_3 x_4 x_1\\
& = x_2^{(1)} x_1^{(1)} x_3^{(1)} x_2^{(2)} x_1^{(2)} x_4^{(1)} x_3^{(2)} x_2^{(3)} x_4^{(2)} x_3^{(3)} x_4^{(3)} x_1^{(3)}.
\end{align*}
The occurrences are given for helping to follow the proofs below.
Both these words represent the path $P_4$ and $\mbox{\texttt{Ev}}$ is a shift of $\mbox{\texttt{Od}}$ by one position to the left. Note also that 
$\mbox{\texttt{Od}}$ contains the factors $x_1^{(2)}x_3^{(1)}x_2^{(2)}$ and $x_3^{(2)}x_2^{(3)}x_4^{(2)}$ while  $\mbox{\texttt{Ev}}$ contains the factors $x_2^{(1)}x_1^{(1)}x_3^{(1)}$ and $x_2^{(3)}x_4^{(2)}x_3^{(3)}$.

\begin{theorem}\label{King4}  The king  graph $\mathrm{King}_{4,n}$ is $3$-representable for all $n \geq 1$. Moreover, there exists a word $Y_n$ representing it such that:
\begin{itemize}
\item if $n$ is odd, then $Y_n$ contains the factors $x_{1,n}^{(2)}x_{3,n}^{(1)} x_{2,n}^{(2)}$ and $x_{3,n}^{(2)} x_{2,n}^{(3)} x_{4,n}^{(2)}$;
\item if $n$ is even, then $Y_n$ contains the factors $x_{2,n}^{(1)} x_{1,n}^{(1)} x_{3,n}^{(1)}$ and $x_{2,n}^{(3)} x_{4,n}^{(2)} x_{3,n}^{(3)}$.
\end{itemize}
\end{theorem}

\begin{proof}
We construct $Y_n$ recursively. Put $V_r=\{x_{1,r},x_{2,r},x_{3,r},x_{4,r}\}$ and write $W_r=Y_n|_{V_r}$ for the restriction to the $r^{\rm th}$ layer (this restriction is independent of later layers). Our induction invariant consists of the two factor conditions in the theorem together with the assertion that, after the second subscript is suppressed, $W_r$ is \texttt{Od} for odd $r$ and \texttt{Ev} for even $r$. In particular, all order relations among occurrences within $W_r$ that are used below are part of the invariant. For $n = 1$, we start with the word
\[
Y_1 = x_{1,1}^{(1)} x_{2,1}^{(1)} x_{1,1}^{(2)} x_{3,1}^{(1)} x_{2,1}^{(2)} x_{1,1}^{(3)} x_{4,1}^{(1)} x_{3,1}^{(2)} x_{2,1}^{(3)} x_{4,1}^{(2)} x_{3,1}^{(3)} x_{4,1}^{(3)},
\]
which coincides, after removing the second subindex, with $\mbox{\texttt{Od}}$ representing $P_4$, and contains the factors  
$x_{1,1}^{(2)} x_{3,1}^{(1)} x_{2,1}^{(2)}$ and $x_{3,1}^{(2)} x_{2,1}^{(3)} x_{4,1}^{(2)}$.

Given a 3-uniform word $Y_{n-1}$ satisfying the induction invariant, we obtain $Y_n$ as follows.
 If $n$ is even, substitute
\begin{eqnarray*}
\mbox{the factor } x_{1,n-1}^{(2)}x_{3,n-1}^{(1)}x_{2,n-1}^{(2)} & \mbox{ by } & x_{2,n}^{(1)}x_{1,n}^{(1)}x_{3,n}^{(1)} x_{1,n-1}^{(2)}x_{3,n-1}^{(1)}x_{2,n-1}^{(2)}x_{2,n}^{(2)}x_{1,n}^{(2)};\\
\mbox{the factor } x_{3,n-1}^{(2)}x_{2,n-1}^{(3)}x_{4,n-1}^{(2)} & \mbox{ by } & x_{4,n}^{(1)}x_{3,n}^{(2)}x_{3,n-1}^{(2)}x_{2,n-1}^{(3)}x_{4,n-1}^{(2)}x_{2,n}^{(3)}x_{4,n}^{(2)}x_{3,n}^{(3)}; \\
x_{4,n-1}^{(3)} & \mbox{ by } & x_{4,n-1}^{(3)}x_{4,n}^{(3)}x_{1,n}^{(3)}.
\end{eqnarray*}

If $n$ is odd, substitute
\begin{eqnarray*}
\mbox{the factor } x_{2,n-1}^{(1)}x_{1,n-1}^{(1)}x_{3,n-1}^{(1)} & \mbox{ by } & x_{1,n}^{(1)}x_{2,n}^{(1)}x_{2,n-1}^{(1)}x_{1,n-1}^{(1)}x_{3,n-1}^{(1)}x_{1,n}^{(2)}x_{3,n}^{(1)}x_{2,n}^{(2)};\\
x_{4,n-1}^{(1)} & \mbox{ by } & x_{1,n}^{(3)}x_{4,n}^{(1)}x_{4,n-1}^{(1)};\\
\mbox{the factor } x_{2,n-1}^{(3)}x_{4,n-1}^{(2)}x_{3,n-1}^{(3)} & \mbox{ by } & x_{3,n}^{(2)}x_{2,n}^{(3)}x_{4,n}^{(2)}x_{2,n-1}^{(3)}x_{4,n-1}^{(2)}x_{3,n-1}^{(3)}x_{3,n}^{(3)}x_{4,n}^{(3)}. 
\end{eqnarray*}

In each case, deleting the old letters from the three replacement strings leaves the twelve new occurrences in exactly the order \texttt{Od} when $n$ is odd and \texttt{Ev} when $n$ is even. The displayed replacements also contain the two factors required for the next step. Hence the induction invariant is preserved, $W_n$ represents $P_4$, and $Y_n$ is 3-uniform. For each new letter, one interval between two consecutive occurrences lies within a single replacement string and contains old letters only from $V_{n-1}$. Consequently, a letter from any earlier layer occurs zero times in that interval and cannot alternate with the new letter. It remains only to verify adjacencies between the new and preceding layers. Write $N_{n-1}(x)=N(x)\cap V_{n-1}$. We consider two cases. \\[-3mm]

\noindent
{\bf Case 1.} Suppose that $n$ is even. By the first replacement and the restriction notation introduced above,
\[
[x_{2,n}^{(1)},x_{2,n}^{(2)}]_{Y_n}|_{V_{n-1}}
=[x_{1,n}^{(1)},x_{1,n}^{(2)}]_{Y_n}|_{V_{n-1}}
=x_{1,n-1}^{(2)}x_{3,n-1}^{(1)}x_{2,n-1}^{(2)}.
\]
Moreover,
\[
x_{1,n-1}^{(3)}<x_{3,n-1}^{(2)}<x_{2,n-1}^{(3)}<x_{2,n}^{(3)}
<x_{3,n-1}^{(3)}<x_{1,n}^{(3)}.
\]
This follows from the structure of $W_{n-1}=\mbox{\texttt{Od}}$ and because the last replacement inserts $x_{1,n}^{(3)}$ after $x_{4,n-1}^{(3)}$. Therefore,
\begin{align*}
N_{n-1}(x_{1,n})&=\{x_{1,n-1},x_{2,n-1}\},\\
N_{n-1}(x_{2,n})&=\{x_{1,n-1},x_{2,n-1},x_{3,n-1}\}.
\end{align*}
Similarly,
\[
[x_{4,n}^{(1)},x_{4,n}^{(2)}]_{Y_n}|_{V_{n-1}}
=[x_{3,n}^{(2)},x_{3,n}^{(3)}]_{Y_n}|_{V_{n-1}}
=x_{3,n-1}^{(2)}x_{2,n-1}^{(3)}x_{4,n-1}^{(2)},
\]
and the contrapositive of Observation~\ref{evident} gives $(x_{2,n-1},x_{4,n})\not\in E$. Since
\[
x_{3,n-1}^{(3)}<x_{4,n-1}^{(3)}<x_{4,n}^{(3)}
\]
and
\[
x_{2,n-1}^{(1)}<x_{3,n}^{(1)}<x_{3,n-1}^{(1)}<x_{2,n-1}^{(2)}
<x_{4,n-1}^{(1)}<x_{3,n}^{(2)},
\]
we have
\begin{align*}
N_{n-1}(x_{3,n})&=\{x_{2,n-1},x_{3,n-1},x_{4,n-1}\},\\
N_{n-1}(x_{4,n})&=\{x_{3,n-1},x_{4,n-1}\},
\end{align*}
as required. \\[-3mm]

\noindent
\textbf{Case 2.} If $n \ge 3$ is odd, then $W_n$ contains the factor
\[
x_{1,n}^{(1)} x_{2,n}^{(1)} x_{2,n-1}^{(1)} x_{1,n-1} ^{(1)} x_{3,n-1} ^{(1)} x_{1,n} ^{(2)} x_{3,n} ^{(1)} x_{2,n} ^{(2)}.
\]
The induction invariant gives the relative order of the old occurrences in $W_{n-1}=\mbox{\texttt{Ev}}$; reading the positions of the inserted letters in the three replacement strings then gives
\[
x_{2,n-1} ^{(2)} < x_{1,n-1} ^{(2)} < x_{1,n} ^{(3)} < x_{3,n-1} ^{(2)} < x_{2,n} ^{(3)} < x_{j,n-1} ^{(3)} \quad \text{for all } j = 1,2,3.
\]
Thus,
\begin{align*}
N_{n-1}(x_{1,n})&=\{x_{1,n-1},x_{2,n-1}\},\\
N_{n-1}(x_{2,n})&=\{x_{1,n-1},x_{2,n-1},x_{3,n-1}\},
\end{align*}
as required.

Similarly, since $W_n$ contains the factor
\[
x_{3,n} ^{(2)} x_{2,n} ^{(3)} x_{4,n} ^{(2)} x_{2,n-1} ^{(3)} x_{4,n-1} ^{(2)} x_{3,n-1} ^{(3)} x_{3,n} ^{(3)} x_{4,n} ^{(3)}
\]
and also
\[
x_{2,n-1}^{(1)}<x_{3,n-1}^{(1)}<x_{3,n}^{(1)}<x_{2,n-1}^{(2)}
<x_{4,n}^{(1)}<x_{4,n-1}^{(1)}<x_{3,n-1}^{(2)},
\]
we have
\begin{align*}
N_{n-1}(x_{4,n})&=\{x_{3,n-1},x_{4,n-1}\},\\
N_{n-1}(x_{3,n})&=\{x_{2,n-1},x_{3,n-1},x_{4,n-1}\}.
\end{align*}
\end{proof}

\subsection{Queen graphs}\label{queen-sec} 

Recall that a \emph{comparability graph} is a graph that admits a transitive orientation, and the \emph{neighbourhood} of a vertex $v$ in a graph is the subgraph induced by the vertices adjacent to $v$. We need the following result.

\begin{theorem}[\cite{KitPya08}]\label{comparability}
If a graph $G$ is word-representable, then the neighbourhood of each vertex in $G$ is a comparability graph.
\end{theorem}

The following theorem shows that the graph $\mathrm{Queen}_{m,n}$ is not word-representable for $m,n \geq 3$.

\begin{theorem}\label{thm-Q33}  We have $\mathcal{R}(\mathrm{Queen}_{m,n}) = \infty$ for $m,n \geq 3$. \end{theorem}

\begin{proof} 
We first show that the neighbourhood of the vertex $x_{2,2}$ in $\mathrm{Queen}_{3,3}$ (depicted in Figure~\ref{queen-neighborhood}) is not a comparability graph (that is, it does not admit a transitive orientation). Indeed, w.l.o.g., assume that the edge $(x_{1,1},x_{1,2})$ is oriented as $x_{1,1}\!\to\! x_{1,2}$. To satisfy transitivity, the following edge orientations are then forced (justifications are given in parentheses):
\begin{itemize} 
\item[] $x_{3,2}\rightarrow x_{1,2}$ (as $x_{1,1}$ and $x_{3,2}$ are not adjacent), 

$x_{1,3}\rightarrow x_{1,2}$ (as $x_{1,3}$ and $x_{3,2}$ are not adjacent), 

$x_{2,3}\rightarrow x_{1,2}$ (as $x_{1,1}$ and $x_{2,3}$ are not adjacent), 

$x_{2,1}\rightarrow x_{1,2}$ (as $x_{1,3}$ and $x_{2,1}$ are not adjacent), 

$x_{2,3}\rightarrow x_{3,3}$ (as $x_{1,2}$ and $x_{3,3}$ are not adjacent),

 $x_{2,1}\rightarrow x_{3,1}$ (as $x_{1,2}$ and $x_{3,1}$ are not adjacent).
\end{itemize}

But then there is no good way to orient the edge $(x_{3,1},x_{3,3})$. Indeed, 
\begin{itemize} 
\item[]
if  $x_{3,1}\rightarrow x_{3,3}$ then $x_{2,1}\rightarrow x_{3,1}\rightarrow x_{3,3}$ while  $x_{2,1}$ and  $x_{3,3}$ are not connected;

if $x_{3,3}\rightarrow x_{3,1}$ then $x_{2,3}\rightarrow x_{3,3}\rightarrow x_{3,1}$ while  $x_{2,3}$ and  $x_{3,1}$ are not connected.  
\end{itemize}

Since $\mathrm{Queen}_{3,3}$  contains a vertex whose neighbourhood is not a comparability graph, $\mathrm{Queen}_{3,3}$ is non-word-representable by Theorem~\ref{comparability} and therefore, $\mathrm{Queen}_{m,n}$ is not word-representable for $m,n\geq 3$ (that is, $\mathcal{R}(\mathrm{Queen}_{m,n})=\infty$ in this case). 
\end{proof}

Since  $\mathcal{R}(\mathrm{Queen}_{1,n})=\mathcal{R}(\mathrm{Queen}_{2,2})=1$ for $n\geq 1$, by Theorem~\ref{thm-Q33}, the only remaining case is determining  $\mathcal{R}(\mathrm{Queen}_{2,n})$ for $n\geq 3$. 

Note that $\mathrm{Queen}_{2,5}$ is not 2-representable, since columns $1$, $3$, and $5$ induce a triangular prism, which is not a circle graph \cite{Kit13}. Also, note that $\mathrm{Queen}_{2,3}$ is the complete graph $K_6$ without two edges, $(x_{1,1},x_{2,3})$ and $(x_{1,3},x_{2,1})$, and thus it can be 2-represented by 
$$x_{1,1}x_{2,3}x_{1,3}x_{2,1}x_{1,2}x_{2,2}x_{2,3}x_{1,1}x_{2,1}x_{1,3}x_{1,2}x_{2,2}.$$
Finally, a  2-representation of $\mathrm{Queen}_{2,4}$ is given by
$$x_{1,1}x_{1,2}x_{2,1}x_{1,3}x_{2,2}x_{1,4}x_{1,1}x_{2,3}x_{1,2}x_{2,4}x_{1,3}x_{1,4}x_{2,1}x_{2,2}x_{2,3}x_{2,4}.$$

\begin{theorem}
The queen graph  $\mathrm{Queen}_{2,n}$ is $3$-representable for all $n\ge 5$.
\end{theorem}

\begin{proof}
Products below denote concatenation in increasing order of the index. Consider the word
\begin{align*}
Q_n={}&x_{1,1}
\left(\prod_{j=1}^{n-1}x_{1,j+1}x_{2,j}\right)x_{2,n}
\left(\prod_{j=1}^{n}x_{1,j}\right)x_{2,1}x_{2,2}\\
&\quad\cdot\left(\prod_{j=1}^{n-2}x_{1,j}x_{2,j+2}\right)
x_{1,n-1}x_{1,n}\left(\prod_{j=1}^{n}x_{2,j}\right).
\end{align*}
This formula makes the endpoints of every factor explicit. Each letter occurs three times. Moreover, the three occurrences of the letters in each row appear in increasing column order, so each row induces a clique $K_n$.

For all $i,j$, the middle occurrence of $x_{1,i}$ lies after the first occurrence and before the second occurrence of $x_{2,j}$, while the third occurrence of $x_{1,i}$ lies before the third occurrence of $x_{2,j}$. If $|i-j|\leq1$, the remaining two comparisons have the required direction, and hence
\[
x_{1,i}^{(1)}<x_{2,j}^{(1)}<x_{1,i}^{(2)}<x_{2,j}^{(2)}<x_{1,i}^{(3)}<x_{2,j}^{(3)}.
\]
Thus $x_{1,i}$ and $x_{2,j}$ alternate. If $i>j+1$, then $x_{2,j}^{(1)}<x_{1,i}^{(1)}<x_{1,i}^{(2)}$, so they do not alternate. If $i<j-1$, then $x_{1,i}^{(2)}<x_{1,i}^{(3)}<x_{2,j}^{(2)}$, and again they do not alternate. These are exactly the cross-row adjacencies of $\mathrm{Queen}_{2,n}$.
\end{proof}

\subsection{Rook  graphs}\label{rook-sec} 

We have that for any $m, n \geq 1$, $\mathcal{R}(\mathrm{Rook}_{m,n}) < \infty$, since any rook  graph is the Cartesian product of two complete graphs, which are 1-representable, and the Cartesian product of two word-representable graphs is always word-representable~\cite{KL15}.

Note that $\mathcal{R}(\mathrm{Rook}_{1,n}) = 1$ for $n \geq 1$, and $\mathcal{R}(\mathrm{Rook}_{2,2}) = \mathcal{R}(C_4) = 2$, where $C_4$ is the 4-cycle.
 Further note that $\mathrm{Rook}_{2,3}$ is not 2-representable, since it is a triangular prism, which is not a circle graph~\cite{Kit13}. Hence, $\mathrm{Rook}_{m,n}$ is not 2-representable for $m \geq 2$ and $n \geq 3$. 

Finally, we note that $\mathrm{Rook}_{3,3} = \mathrm{TGr}_{3,3}$, and~\cite{AKP2025} provides a 3-representation of it:
 $$abcdefgadhigbcaehbfdeighcfi,$$ where, for brevity, we rename the vertices $x_{1,1}$, $x_{2,1}$, $x_{3,1}$, $x_{1,2}$, $x_{2,2}$, $x_{3,2}$, $\ldots$ as $a$, $b$, $c$, $d$, $e$, $f$, $\ldots$, respectively.

\begin{theorem}
The rook  graph  $\mathrm{Rook}_{2,n}$ is $3$-representable for all $n\ge 3$.
\end{theorem}

\begin{proof}
Products denote concatenation in increasing order. Consider the 3-uniform word
\[
R_n=\left(\prod_{i=1}^{n}x_{1,i}x_{2,i}\right)
\left(\prod_{i=1}^{n}x_{1,i}\right)
\left(\prod_{i=1}^{n}x_{2,i}x_{1,i}\right)
\left(\prod_{i=1}^{n}x_{2,i}\right).
\]
In each row, all three copies occur in increasing column order, so each row induces a clique $K_n$. If $i=j$, then
\[
x_{1,i}^{(1)}<x_{2,i}^{(1)}<x_{1,i}^{(2)}<x_{2,i}^{(2)}<x_{1,i}^{(3)}<x_{2,i}^{(3)},
\]
and hence $x_{1,i}$ alternates with $x_{2,i}$. If $i>j$, then
$x_{2,j}^{(1)}<x_{1,i}^{(1)}<x_{1,i}^{(2)}$, whereas if $i<j$, then
$x_{1,i}^{(2)}<x_{1,i}^{(3)}<x_{2,j}^{(2)}$. Thus cross-row letters in different columns do not alternate. These are exactly the adjacencies of $\mathrm{Rook}_{2,n}$.
\end{proof}

Finally, a computer search found no 3-uniform word representing $\mathrm{Rook}_{3,4}$, suggesting that $\mathcal{R}(\mathrm{Rook}_{3,4})\geq4$ and hence that $\mathcal{R}(\mathrm{Rook}_{3,m})\geq4$ for $m\geq4$. We do not regard this computation as a proof; establishing or refuting the suggested lower bound remains open.
 
\subsection{Bishop  graphs}\label{bishop-sec} 

Every bishop graph is the disjoint union of the subgraphs induced by the white and black squares. These two components need not contain one another as induced subgraphs; for example, the two components of $\mathrm{Bishop}_{3,3}$ are non-isomorphic. We therefore treat the components separately when needed. In particular, concatenating 2-representations of the two components gives a 2-representation of their disjoint union.

Now, $\mathcal{R}(\mathrm{Bishop}_{1,1}) = 1$ and $\mathcal{R}(\mathrm{Bishop}_{1,n}) = 2$ for $n \geq 2$, since $\mathrm{Bishop}_{1,n}$ is an independent set, which
is clearly 2-representable.
%can be represented by the word $x_{1,1}x_{1,1}x_{1,2}x_{1,2} \ldots x_{1,n}x_{1,n}$. 
Also, $\mathcal{R}(\mathrm{Bishop}_{2,n}) = 2$ for $n \geq 2$, since $\mathrm{Bishop}_{2,n}$ is a union of two disjoint paths $P_n$ and can be represented as the concatenation of the two 2-uniform words representing paths (see~\cite{KL15} for a method to represent trees).
%:
%\[
%x_{1,1}x_{2,2}x_{1,1}x_{1,3}x_{2,2}x_{2,4}x_{1,3}x_{1,5}x_{2,4} \ldots,
%\]
%\[
%x_{2,1}x_{1,2}x_{2,1}x_{2,3}x_{1,2}x_{1,4}x_{2,3}x_{2,5}x_{1,4} \ldots.
%\]

Furthermore, observe that a bishop  graph is structurally similar to a rook  graph and can be viewed as a disjoint union of rook  graphs corresponding to chess played on a “rhombus-shaped” chessboard rather than a rectangular one. This observation shows that $\mathcal{R}(\mathrm{Bishop}_{m,n}) < \infty$, since $\mathrm{Bishop}_{m,n}$ is a disjoint union of two subgraphs of rook  graphs, which are word-representable, and the following more specific result holds:

\begin{proposition}\label{bishop-bounds}
For $2\leq m \leq n$, we have
\begin{small}
\[
\max\left\{
\mathcal{R}\left(\mathrm{Rook}_{\left\lfloor \frac{m}{2} \right\rfloor,\; \left\lceil \frac{m}{2} \right\rceil + 1}\right),
\mathcal{R}\left(\mathrm{Rook}_{\left\lfloor \frac{m+1}{2} \right\rfloor,\; \left\lfloor \frac{m+1}{2} \right\rfloor}\right)
\right\}
\leq
\mathcal{R}(\mathrm{Bishop}_{m,n})
\]
and
\[
\mathcal{R}(\mathrm{Bishop}_{m,n})
\leq
\mathcal{R}(\mathrm{Rook}_{n,n}).
\]
\end{small}
\end{proposition}

\begin{proof} The proof follows from the fact that, for $m \leq n$, the graph $\mathrm{Bishop}_{m,n}$ is an induced subgraph of the graph $\mathrm{Rook}_{n,n}$, and also, $\mathrm{Bishop}_{m,n}$ contains both 
$\mathrm{Rook}_{\left\lfloor \frac{m}{2} \right\rfloor,\; \left\lceil \frac{m}{2} \right\rceil + 1}$ 
and 
$\mathrm{Rook}_{\left\lfloor \frac{m+1}{2} \right\rfloor,\; \left\lfloor \frac{m+1}{2} \right\rfloor}$ 
as induced subgraphs. \end{proof}

Note that $\mathrm{Bishop}_{4,4}$ contains a triangular prism, which is known to be not 2-representable \cite{Kit13}, and hence $\mathcal{R}(\mathrm{Bishop}_{4,4}) \geq 3$. The following theorem completes the characterization of all 2-representable (i.e., circle) bishop graphs.

\begin{theorem}
The graph $\mathrm{Bishop}_{3,n}$ is $2$-representable for all $n \ge 3$.
\end{theorem}

\begin{proof}
Put $a_j=x_{1,j}$, $b_j=x_{2,j}$, and $c_j=x_{3,j}$. We use the following elementary insertion rule. If $pq$ is a factor of a 2-uniform word and $z$ is new, replacing $pq$ by $zpqz$ makes $z$ alternate precisely with $p$ and $q$: exactly one occurrence of each lies between the two occurrences of $z$, and no occurrence of any other letter does.

First consider the color component containing $a_1$ and $c_1$. Begin with
\[
a_1c_1c_1a_1,
\]
which contains the factor $c_1a_1$. The other color component is initialized, for $n\geq2$, by
\[
a_2b_1a_2c_2b_1c_2,
\]
which represents the path $a_2b_1c_2$ and contains the factor $a_2c_2$; for $n=1$, it is represented by $b_1b_1$.

We now give the common recursive step. Suppose that the columns through an outer-vertex column $j$ have been represented and that the word contains a factor
$F\in\{c_ja_j,a_jc_j\}$. To add $b_{j+1}$, replace $F$ by $b_{j+1}Fb_{j+1}$. This gives exactly the edges from $b_{j+1}$ to $a_j$ and $c_j$. If $j+1=n$, the construction is complete. Otherwise, the new factor contains two disjoint adjacent pairs: let $F_1$ be the pair consisting of $b_{j+1}$ and $c_j$, in their current order, and let $F_2$ be the pair consisting of $a_j$ and $b_{j+1}$. These pairs occur as $F_1F_2$ or $F_2F_1$.

To add the two vertices in column $j+2$, replace
\[
F_1\quad\hbox{by}\quad a_{j+2}F_1a_{j+2},
\qquad
F_2\quad\hbox{by}\quad c_{j+2}F_2c_{j+2}.
\]
The insertion rule gives exactly the new edges: $a_{j+2}$ is adjacent to $b_{j+1}$ and $c_j$, while $c_{j+2}$ is adjacent to $a_j$ and $b_{j+1}$. The two new letters do not alternate with each other, and the resulting word contains $a_{j+2}c_{j+2}$ or $c_{j+2}a_{j+2}$ as a factor, so the invariant is restored. Repeating the two steps gives a 2-representation of each color component of $\mathrm{Bishop}_{3,n}$. Concatenating the two words gives the required 2-representation of the whole graph.
\end{proof}

\subsection{Knight  graphs}\label{knight-sec} 

Knight graphs are bipartite and hence word-representable (any 3-colourable graph is word-representable~\cite{KL15}), so the representation number of any such graph is finite. In what follows, we characterize all knight graphs that are circle graphs; equivalently, apart from complete graphs, these are precisely the knight graphs with representation number~2.

Note that $\mathcal{R}(\mathrm{Knight}_{1,1}) = 1$ and $\mathcal{R}(\mathrm{Knight}_{1,n}) = 2$ for $n \geq 2$, since $\mathrm{Knight}_{1,n}$ is an  independent set for all $n \geq 1$. Also, $\mathcal{R}(\mathrm{Knight}_{2,n}) = 2$ for any $n \geq 1$, since this graph is a disjoint union of four 2-representable paths. Moreover, $\mathcal{R}(\mathrm{Knight}_{3,3}) = 2$, since $\mathrm{Knight}_{3,3}$ is a disjoint union of the 2-representable cycle $C_8$ and isolated vertex $x_{2,2}$. A 2-representation of $\mathrm{Knight}_{3,3}$ is
\[
x_{2,2}x_{2,2}x_{1,1}x_{3,2}x_{2,3}x_{1,1}x_{3,1}x_{2,3}x_{1,2}x_{3,1}x_{3,3}x_{1,2}x_{2,1}x_{3,3}x_{1,3}x_{2,1}x_{3,2}x_{1,3}.
\]
Computer experiments produced the following 2-representation of $\mathrm{Knight}_{3,4}$ and suggest that it is unique up to reversal and cyclic shift:
\begin{align*}
&x_{1,1}x_{3,2}x_{2,3}x_{1,1}x_{3,1}x_{2,3}x_{1,2}x_{3,1}x_{2,4}\\
&~~~~~~~~~~~~~~~~~~~~x_{3,3}x_{1,2}x_{2,1}x_{1,4}x_{3,3}x_{2,2}x_{1,4}x_{3,4}x_{2,2}x_{1,3}x_{3,4}x_{2,1}x_{3,2}x_{1,3}x_{2,4}
\end{align*}
The displayed word proves that $\mathcal{R}(\mathrm{Knight}_{3,4}) = 2$; only the uniqueness assertion is based on computation.

Finally, $\mathrm{Knight}_{3,5}$ and $\mathrm{Knight}_{4,4}$  are not 2-representable, because after removing the vertices of $\{x_{2,1}, x_{2,5}, x_{3,5}\}$ and $\{x_{2,1}, x_{4,1}, x_{4,3}, x_{4,4}\}$, respectively, one obtains subdivisions of the %Bicolored Wheel
 graph $BW_3$ (shown in Figure~\ref{BW3}) % on 6 vertices (sometimes called the ``Butterfly-Wheel''). According to 
and by Bouchet’s theorem~\cite{Bouchet}, circle graphs cannot contain such induced subgraphs.

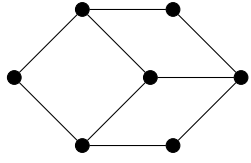
\begin{figure}
\begin{center}
\begin{tikzpicture}[scale=0.3]
			
\draw (2,5) node[scale=0.5, circle, draw, fill=black] (node1) {};
\draw (5,8) node [scale=0.5, circle, draw, fill=black](node2){};
\draw (9,8) node [scale=0.5, circle, draw, fill=black](node3){};
\draw (12,5) node [scale=0.5, circle, draw, fill=black](node4){};
\draw (9,2) node [scale=0.5, circle, draw, fill=black](node5){};
\draw (5,2) node [scale=0.5, circle, draw, fill=black](node6){};
\draw (8,5)node [scale=0.5, circle, draw, fill=black](node7){};
			
\draw (node1)--(node2);
\draw (node2)--(node3);
\draw (node3)--(node4);
\draw (node4)--(node5);
\draw (node5)--(node6);
\draw (node6)--(node7);
\draw (node6)--(node1);
\draw (node2)--(node7);
\draw (node7)--(node4);
\end{tikzpicture}

\end{center}
\caption{The graph $BW_3$} \label{BW3}
\end{figure}

\section{Open problems}

In this paper, we provide a complete characterization of the representation number of $\mathrm{Queen}_{m,n}$. We leave the following as open problems regarding finding the representation number, which is finite:
\begin{itemize}
\item $\mathrm{King}_{m,n}$ for $n\geq m \geq 5$, which is conceivable to have representation number 3, though we are unable to prove or disprove this;
\item $\mathrm{Rook}_{m,n}$ for $n\geq4$ and $3\leq m\leq n$ (computer experiments suggest that $\mathcal{R}(\mathrm{Rook}_{3,n})\geq4$ for $n\geq4$, but this lower bound has not been proved);
\item $\mathrm{Bishop}_{m,n}$ for $n \geq m \geq 4$ (possibly using Proposition~\ref{bishop-bounds}; recall that $\mathcal{R}(\mathrm{Bishop}_{4,n}) \geq 3$ for $n \geq 4$);
\item $\mathrm{Knight}_{m,n}$ for $n\geq m \geq 3$, which is also conceivable to have representation number 3, but again, we cannot prove or disprove it.
\end{itemize}

\section*{Acknowledgments} The authors thank the anonymous reviewers for their careful reading and helpful comments, which improved the presentation of the paper. The work of the second author was supported by the research project of the Sobolev Institute of Mathematics (project FWNF-2022-0019).


\begin{thebibliography}{20}

\bibitem{AKM} P. Akrobotu, S. Kitaev, and Z. Mas\'{a}rov\'{a}. On word-representability of polyomino triangulations. {\em Siberian Adv. in Math.}, {\bf 25(1)} (2015) 1--10.

\bibitem{AKP2025} N.S. Alshammari, S. Kitaev, and A. Pyatkin. On the representation number of grid graphs and cylindric grid graphs, 	arXiv:2507.16469 (2025).

\bibitem{Bouchet} A. Bouchet. Circle Graph Obstructions. {\em J. Combin. Theory Ser. B} {\bf 60(1)} (1994) 107--144.

\bibitem{Bro18} B. Broere. Word-Representable Graphs, (Master thesis), Radboud University, Nijmegen, 2018.

\bibitem{BroZan19} B. Broere and H. Zantema. The $k$-dimensional cube is $k$-representable, {\em J. Autom. Lang. Comb.} {\bf 24 (1)} (2019) 3--12.

\bibitem{Bur2016} P.A. Burchett. $k$-tuple domination on the bishop  graph. {\em Util. Math.} {\bf 101} (2016) 351--358.

\bibitem{BLL} P.A. Burchett, D. Lane, and J.A. Lachniet, $k$-tuple and $k$-domination on the rook  graph and other results, {\em Congr. Numer.} {\bf 199} (2009) 187--204.

\bibitem{Cibu2024} N.~Cibu, K.~Ding, S.~DiSilvio, S.~Kononova, C.~Lee, R.~Morrison, and K.~Singal,
The gonality of chess graphs, arXiv:2403.03907, 2024.

%\bibitem{Glen} M. Glen. Software available at  \verb>spider-v.science.strath.ac.uk/> \verb>sergey.kitaev/word-representable-graphs.html>

\bibitem{Fleisch} P. Fleischmann, L. Haschke, T. L\"ock, and D. Nowotka. Word-representable graphs from a word’s perspective. In {\em Int'l Conf. on Curr. Trends in Theory and Pract. of Comput. Sci.}, pages 255--268. Springer, 2024.
	
\bibitem{GKP18} M. Glen, S. Kitaev, and A. Pyatkin. On the representation number of a crown graph, {\em Discrete Appl. Math.} {\bf 244} (2018), 89--93. 

\bibitem{Haynes1998} T.~W. Haynes, S.~T. Hedetniemi, and P.~J. Slater (eds.), \emph{Domination in Graphs: Advanced Topics}, Marcel Dekker, New York, 1998.

\bibitem{HHHH} J. T. Hedetniemi, K. D. Hedetniemi, S. M. Hedetniemi, S. T. Hedetniemi. On the domatic numbers of queens graphs. {\em Congr. Numer.} {\bf 235} (2025), 5--21.

\bibitem{HHMO24} Z. Hefty, P. Horn, C. Muir, and A. Owens. Word-Representable Graphs: Orientations, Posets, and Bounds, {\em Electron. J. Combin.} {\bf 31(4)} (2024), \#P4.2.

\bibitem{Kit13} S. Kitaev. On graphs with representation number 3, {\em J. Autom. Lang. Combin.} {\bf 18(2)} (2013) 97--112.
	
\bibitem{KL15} S. Kitaev and V. Lozin. Words and Graphs, {\em Springer}, 2015.

\bibitem{KitPya08} S. Kitaev and A. Pyatkin. On representable graphs,  {\em J. Autom. Lang. Comb.} {\bf 13} (2008), no. 1, 45--54.

\bibitem{KP18} S.V. Kitaev and A.V. Pyatkin.  Word-Representable Graphs: a Survey, {\em J. Appl. and Industr. Math.} {\bf 12} (2018), no.~2, 278--296. 

\bibitem{KitSei} S. Kitaev, S. Seif. Word problem of the Perkins semigroup via directed acyclic graphs, {\em Order} {\bf 25} (2008)
  3, 177--194.

\bibitem{Ma} D. Ma. The crossing number of the strong product of two paths. {\em Australas. J. Combin.} {\bf 68} (2017) 35--47.
 
\bibitem{SowndaryaNaidu2020} K.~S.~P. Sowndarya and Y.~Lakshmi Naidu. Perfect Domination Separation on Square Chessboard,
\emph{Malaya Journal of Matematik} {\bf 8(4)} (2020) 1497--1501.

\bibitem{VWSFN} R. J. Vargas, A. Waldron, A. Sharma,  R. Fl\'{o}rez, D. A. Narayan. Leverage centrality of knight  graphs and Cartesian products of regular graphs and path powers. {\em Involve} {\bf 10} (2017), no. 4, 583--592.	

\bibitem{Watkins1996} M. E. Watkins, \emph{Across the Board: The Mathematics of Chessboard Problems},
Princeton University Press, 2004.

\bibitem{Wilf} H. Wilf. The Problem of the Kings. {\em Electron. J. Combin.} {\bf 2} (1995), \#R3.	
		
	\end{thebibliography}
\end{document}